\documentclass{amsart}
\usepackage{graphicx}
\newtheorem{thm}{Theorem}[section]

\newtheorem{lem}[thm]{Lemma}

\theoremstyle{definition}

\theoremstyle{remark}
\newtheorem{rem}[thm]{Remark}
\numberwithin{equation}{section}

\begin{document}

\title[Skew $N$-Derivations on Semiprime Rings]{Skew $N$-Derivations on Semiprime Rings}%
\author{Xiaowei Xu}%
\address{College of Mathematics,
Jilin University, Changchun 130012,
P.R.China}%
\email{xuxw@jlu.edu.cn}%

\author{Yang Liu}%
\address{College of Mathematics,
Jilin University, Changchun 130012,
P.R.China}%
\email{liuyang2000@jlu.edu.cn}%

\author{Wei Zhang}%
\address{College of Mathematics,
Jilin University, Changchun 130012,
P.R.China}%
\email{107444672@qq.com}%

\thanks{This work was completed with the support of  the NNSF of China (No. 10871023 and No.11071097),
  211 Project, 985 Project and the Basic Foundation for Science Research from Jilin University.}%

\begin{abstract}
For a ring $R$ with an automorphism $\sigma$, an $n$-additive
mapping $\Delta:R\times R\times\cdots \times R \rightarrow R$  is
called a skew $n$-derivation with respect to $\sigma$ if it is
always a $\sigma$-derivation of $R$ for each argument. Namely, it is
always a $\sigma$-derivation of $R$ for the argument being left once
$n-1$ arguments are fixed by $n-1$ elements in $R$. In this short
note, starting from Bre\v{s}ar Theorems, we prove that a skew
$n$-derivation ($n\geq 3$) on a semiprime ring $R$ must map into the
center of $R$.
\end{abstract}
\maketitle
\section{Introduction}
Let $R$ be a ring with an automorphism $\sigma$. Recall that an
additive mapping $\mu: R\rightarrow R$  is called a
$\sigma$-derivation if $\mu(xy)=\sigma(x)\mu(y)+\mu(x)y$ holds for
all $x,y \in R$. An $n$-additive mapping $$\Delta:R\times
R\times\cdots \times R \rightarrow R$$ (i.e., additive in each
argument) is called a skew $n$-derivation  with respect to $\sigma$
if it is always a $\sigma$-derivation of $R$ for the argument being
left once $n-1$ arguments are fixed by $n-1$ elements in $R$.
Namely, once $a_1,\ldots,a_{i-1},a_{i+1},\ldots,a_n \in R$ are
fixed, then for all $x_i,y_i\in R$, both
$$\Delta(a_1,\ldots,x_i+y_i,\ldots,a_n)=\Delta(a_1,\ldots,x_i,\ldots,a_n)+\Delta(a_1,\ldots,y_i,\ldots,a_n)$$
and
$$\Delta(a_1,\ldots,x_iy_i,\ldots,a_n)=\Delta(a_1,\ldots,x_i,\ldots,a_n)y_i+\sigma(x_i)\Delta(a_1,\ldots,y_i,\ldots,a_n)$$
always hold.

Note that the skew derivation is an ordinary derivation when
$\sigma$ is the identity map $1_R$. Naturally  a skew $n$-derivation
with respect to the identity map $1_R$ is also called an
$n$-derivation.

In order to illustrate  the results  in the literatures focused on
this area clearly we will introduce some concepts related to skew
$n$-derivations although the results in this note  hold for
arbitrary skew $n$-derivations on prime and semiprime rings. A skew
$n$-derivation $\Delta$ is called permuting or symmetric if
$$\Delta(x_1,x_2,\ldots,x_n)=\Delta(x_{\pi(1)},x_{\pi(2)},\ldots,x_{\pi(n)})$$
holds for all $x_1,x_2,\ldots,x_n\in R$ and  $\pi\in S_n$ the
symmetric group of degree $n$. The function $\delta:R\rightarrow R$
defined by $\delta(x)=\Delta(x,x,\ldots,x)$ is called the trace of
$\Delta$. A skew 2-derivation with respect to the automorphism
$\sigma$ is also called a $\sigma$-biderivation.  Naturally a
2-derivation is called a biderivation.  Generalized $n$-derivations
on rings can be defined similarly (see
\cite{bresar-1995-JAlgebra-prime-genera} for the definitions of
generalized biderivations).

A ring $R$ is called prime if $aRb\neq0$ for all $a,b\in
R\backslash\{0\}$. A ring $R$ is called semiprime if $aRa\neq0$ for
all $0\neq a\in R$. For a semiprime ring $R$, denote its extended
centroid by $C$ and its symmetric Martindale ring of quotients by
$Q_s$ (see \cite{book} for reference). Particularly the extended
centroid of a prime ring is a field. Denote the center of $R$ by
$Z(R)$. An automorphism $\sigma$ of a semiprime ring $R$ is called
$X$-inner if there exists  an invertible element $p\in Q_s$ such
that $\sigma(x)=pxp^{-1}$ holds for all $x\in R$. Otherwise $\sigma$
is called $X$-outer. For $a,b\in R$, write the commutator $ab-ba$ of
$a$ and $b$ by $[a,b]$. We will always use the commutator formulas
$[a,bc]=b[a,c]+[a,b]c$ and $[ab,c]=a[b,c]+[a,c]b$ for $a,b,c\in R$.

The notion of a symmetric biderivation had been introduced by Maksa
\cite{Maksa-1980-GlasMS} in 1980. In 1989, Vukman
\cite{Vukman-1989-AequaMath} initiated the research of biderivations
on prime and semiprime rings. He extended classical  Posner Theorem
\cite{Posner-1957-PAMS} to symmetric biderivations in  prime and
semiprime rings. Thereafter many literatures were focused on
biderivations of prime and semiprime rings (see
\cite{Ali-2006-Aligarh-prime,Argac-2006-AlgebraColloq-prime,Argac-1992-SymmSci-prime,Argac-1996-PureAppl-prime,
Ashraf-1999-Rend-prime,Ashraf-1997-8-Aligarh-prime,bresar-1995-JAlgebra-prime,bresar-1995-JAlgebra-prime-genera,
Ceran-2009-Algebras-prime,Deng-1996-JMathResExp-prime,Deng-1997-InternatJMath-prime,Hongan-2000-CommAlg-prime,
Huang-2007-study-prime,Kannappan-1994-Glas-prime,Khan-2003-Southeas-prime,
 Muthanna-2007-Aligarh-prime,Vukman-1990-AequaMath-prime,
WangY-2002-CommAlg-prime,Yenigul-1993-MathJOkayamaMath-prime,zhan-2003-Qufu-prime}
for reference). Among these papers, the most important results are
due to Bre\v{s}ar
\cite{1994-Bresar-ProcAMS,bresar-1995-JAlgebra-prime,bresar-1995-JAlgebra-prime-genera}.
He gave the construction of biderivations on semiprime rings. In
\cite{Fosner-submit-to-algebraCollo} Fo\v{s}ner studied the
structure of generalized $\alpha$-derivations of prime and semiprime
rings.

{\bf Bre\v{s}ar Theorem  (\cite[Theorem 4.1]{1994-Bresar-ProcAMS})}
Let $R$ be a semiprime ring, and let $B:R\times R\rightarrow R$ be a
biderivation. Then there exist an idempotent $\varepsilon\in C$ and
an  element $\mu\in C$ such that the algebra $(1-\varepsilon)R$ is
commutative and $\varepsilon B(x,y)=\mu\varepsilon[x,y]$ for all
$x,y\in R$. Particularly if $R$ is a noncommutative  prime ring,
then there exists $\lambda\in C$ such that $B(x,y)=\lambda[x,y]$ for
all $x,y\in R$.

In \cite{bresar-1995-JAlgebra-prime-genera} skew biderivations and
inner generalized biderivations on prime rings were also
characterized. So almost all results appearing in the literatures
listed above can be implied by Bre\v{s}ar Theorems
\cite{1994-Bresar-ProcAMS,bresar-1995-JAlgebra-prime-genera}.

In 2007, Jung and Park \cite{Jung-2007-BKMS-3} considered permuting
3-derivations on prime and semiprime rings and obtained the
following results:

{\bf Theorem A (Jung and Park, \cite[Theorem
2.3]{Jung-2007-BKMS-3})} Let $R$ be a noncommutative 3-torsion free
semiprime ring and let $I$ be a nonzero two-sided ideal of $R$.
Suppose that there exists a permuting 3-derivation $\Delta:R\times
R\times R\rightarrow R$ such that  $\delta$  is centralizing on $I$
($[\delta(x),x]\in Z(R), x\in I$), where $\delta$ is the trace of
$\Delta$. Then $\delta$ is commuting on $I$ ($[\delta(x),x]=0, x\in
I$).

{\bf Theorem B (Jung and Park, \cite[Theorem
2.4]{Jung-2007-BKMS-3})} Let $R$ be a noncommutative 3!-torsion free
prime ring and let $I$ be a nonzero two-sided ideal of $R$. Suppose
that there exists a nonzero permuting 3-derivation $\Delta:R\times
R\times R\rightarrow R$ such that $\delta$  is centralizing on $I$,
where $\delta$ is the trace of $\Delta$. Then $R$ is commutative.

Park \cite{Park-2007-JKS-4} obtained the similar results for
permuting 4-derivations on prime and semiprime rings. Furthermore in
2009, Park \cite{Park-2009-JCCM-4} considered permuting
$n$-derivations on prime and semiprime rings.

In this short note, starting from Bre\v{s}ar Theorems (\cite[Theorem
3.1 and 4.1]{1994-Bresar-ProcAMS}), we prove that an arbitrary skew
$n$-derivation ($n\geq 3$) on a semiprime ring $R$ must map into the
center of $R$. As a corollary, we obtain that  an arbitrary skew
$n$-derivation ($n\geq 3$) on a noncommutative prime ring $R$ must
be zero. These results can reveal the reason why Theorem A, B and
results in the literatures \cite{Park-2007-JKS-4,Park-2009-JCCM-4}
hold.

\section{Main result}
This short note depends heavily on  Bre\v{s}ar Theorems
\cite[Theorem 3.1 and 4.1]{1994-Bresar-ProcAMS}. In view of their
proofs, we give a very mild modification of these two theorems in
order to apply them better. The proof of \cite[Theorem
3.1]{1994-Bresar-ProcAMS} implies its following form.
\begin{rem} \normalfont{\bf(Bre\v{s}ar, \cite[Theorem 3.1]{1994-Bresar-ProcAMS})} Let $S$ be a set and $R$ be a semiprime ring. If
functions $f$ and $g$ of $S$ into $R$ satisfy that
\begin{center}
$f(s)xg(t)=\xi g(s)xf(t)$ ~for all ~$s,t\in S$, $x\in R$,
\end{center}
where $\xi\in C$ is an invertible element, then there exist
idempotents $\varepsilon_1,\varepsilon_2,\varepsilon_3\in C$ and an
invertible element $\lambda\in C$ such that
$\varepsilon_i\varepsilon_j$, for $i\neq j$,
$\varepsilon_1+\varepsilon_2+\varepsilon_3=1$, and
$$\varepsilon_1f(s)=\lambda\varepsilon_1g(s), ~\varepsilon_2g(s)=0,~\varepsilon_3f(s)=0, ~(1-\xi)\varepsilon_1f(s)=0$$
holds for all $s\in S$.
 \label{remark-bresar-semi}
\end{rem}

\begin{proof} It is a small modification of the proof of \cite[Theorem
3.1]{1994-Bresar-ProcAMS}. Define $\varphi: E\rightarrow R$ by
$$\varphi\Big(\varepsilon_1\Big(\sum_{i=1}^nx_if(s_i)y_i\Big)+(1-\varepsilon_1)r
\Big)=\xi\varepsilon_1\Big(\sum_{i=1}^nx_ig(s_i)y_i\Big)+(1-\varepsilon_1)r,$$
and then add $\xi$ in correspondent formulas of the proof of
\cite[Theorem 3.1]{1994-Bresar-ProcAMS}. At last
$(1-\xi)\varepsilon_1f(S)=0$ can be deduced from
$$\Big((1-\xi)\varepsilon_1f(t)\Big)R\Big((1-\xi)\varepsilon_1f(t)\Big)=0$$ for all $t\in
S$.
\end{proof}

 A modification of  the proof for \cite[Theorem
4.1]{1994-Bresar-ProcAMS} will give the following remark.

\begin{rem} Let $R$ be a semiprime ring with an automorphism $\sigma$, and let
$B:R\times R\rightarrow R$ be a $\sigma$-biderivation. Then there
exist    idempotents $\varepsilon_1,\varepsilon_2,\varepsilon_3\in
C$ and invertible elements $p\in Q_s$, $\lambda\in C$
\label{remark-biderivation} such that\begin{itemize}
\item $\varepsilon_1+\varepsilon_2+\varepsilon_3=1$,
$\varepsilon_1\varepsilon_2=\varepsilon_1\varepsilon_3=\varepsilon_2\varepsilon_3=0$,
\item $\varepsilon_1B(x,y)=\varepsilon_1p[x,y], ~\varepsilon_2B(x,y)=0$,
and $\varepsilon_3[x,y]=0$ for all $x,y\in R$.
\end{itemize}
\end{rem}

\begin{proof} By \cite[Lemma 2.3]{bresar-1995-JAlgebra-prime-genera} we have
that for all $x,y,z,u,v\in R$
\begin{align}
B(x,y)z[u,v]=[\sigma(x),\sigma(y)]\sigma(z)B(u,v).\label{equation-remark-skewbiderivation-lemma2.3}
\end{align}
If $\sigma$ is $X$-outer then  for fixed   $x,y,u,v$ we deduce that
$B(x,y)z[u,v]=[\sigma(x),\sigma(y)]z_1B(u,v)$ holds for all
$z,z_1\in R$ by Kharchenko Theorem (\cite[Theorem
2]{Kharchenko-1991-SMJ-auto-semi}). Moreover $B(x,y)z[u,v]=0$ holds
for all $x,y,z,u,v\in R$. So by \cite[Theorem 2.3.9 and Lemma
2.3.10]{book} there exists an idempotent $\varepsilon\in C$ such
that $\varepsilon B(x,y)=(1-\varepsilon)[x,y]=0$ holds for all
$x,y\in R$. Setting $p=1$, $\varepsilon_1=0$,
$\varepsilon_2=\varepsilon$, $\varepsilon_3=1-\varepsilon$, we get
the conclusion in this case. If $\sigma$ is $X$-inner then there
exists an invertible element $p\in Q_s$ such that
$\sigma(x)=pxp^{-1}$ for all $x\in R$. Now observing
\eqref{equation-remark-skewbiderivation-lemma2.3} we get that for
all $x,y,z,u,v\in R$
$$p^{-1}B(x,y)z[u,v]=[x,y]zp^{-1}B(u,v).$$
Following the proof of  \cite[Theorem 4.1]{1994-Bresar-ProcAMS} we
complete the proof. \end{proof}

\begin{rem} From the proof of Remark \ref{remark-biderivation}
for different   $\sigma$-biderivations $B_1, \ldots,B_t$, the
invertible element $p$ is same when $\sigma$ is $X$-inner. We can
set $p=1$ when $\sigma$ is $X$-outer. So $p$ could be thought as
same for different $\sigma$-biderivations.
\end{rem}

 Now we need some lemmas. Lemma \ref{remark-bresar-ours}
and \ref{lemma1} are used to prove Lemma \ref{lemma-main}.  Lemma
\ref{lemma-main} is crucial in the proof of Theorem
\ref{theorem-main}. The proofs are elementary  computation.
\begin{lem}
Let $R$ be a semiprime ring  and $a\in R$. Then $[a,[a,x]]=0$ holds
for all $x\in R$ if and only if $a^2, ~2a\in Z(R)$.
\label{remark-bresar-ours}
\end{lem}

\begin{proof} We only deal with the `` only if " part because the
other part is obvious.  For all $x,y\in R$
\begin{align}\begin{array}{rl}
0=\big[a,[a,xy]\big]&=x\big[a,[a,y]\big]+[a,x][a,y]+[a,x][a,y]+\big[a,[a,x]\big]y\\
&=2[a,x][a,y].
\end{array}
\label{equation-remark2-1}
\end{align}
Putting $x=yz$ in \eqref{equation-remark2-1} and applying
\eqref{equation-remark2-1} we have $[2a,y]R[2a,y]=0$ holds for all
$y\in R$. Then $2a\in Z(R)$ since $R$ is semiprime. Moreover for any
$x\in R$, we obtain
$$0=[a,[a,x]]=a^2x+xa^2-2axa=a^2x+xa^2-x(2a^2)=a^2x-xa^2.$$
So $a^2\in Z(R)$. \end{proof}

\begin{lem}Let $R$ be a semiprime ring with extended centroid $C$ and $a, b\in R$. Then
$[a,[b,x]]=0$ holds for all $x\in R$ if and only if there exist
\label{lemma1}
  idempotents $\varepsilon_1,
\varepsilon_2, \varepsilon_3 \in C$  and  an invertible element
$\lambda \in C$ such that
\begin{itemize}
\item
$\varepsilon_1+\varepsilon_2+\varepsilon_3=1$,
$\varepsilon_1\varepsilon_2=\varepsilon_1\varepsilon_3=\varepsilon_2\varepsilon_3=0$
and
\item $\varepsilon_1a-\lambda\varepsilon_1b, ~\varepsilon_2a,
~\varepsilon_3b, ~2\varepsilon_1b, ~\varepsilon_1b^2 \in C$.
\end{itemize}
\end{lem}

\begin{proof} The `` if " part can be checked by direct computation.
Now we consider the `` only if " part. For any $x,y\in R$
\begin{align}\begin{array}{rl}
0=[a,[b,xy]]&=x\big[a,[b,y]\big]+[a,x][b,y]+[b,x][a,y]+\big[a,[b,x]\big]y\\
&=[a,x][b,y]+[b,x][a,y].
\end{array}
\label{equation-lemma0-1}
\end{align}
 Putting $x=xz$ in
\eqref{equation-lemma0-1} and applying \eqref{equation-lemma0-1} we
have that
\begin{align}
[a,x]z[b,y]+[b,x]z[a,y]=0 \label{equation-lemma0-main}
\end{align}
holds for all $x,y,z\in R$. By Remark \ref{remark-bresar-semi} there
exist idempotents $\varepsilon_1, \varepsilon_2, \varepsilon_3 \in
C$ and an invertible element  $\lambda\in C$ such that
\begin{itemize}
\item $\varepsilon_1+\varepsilon_2+\varepsilon_3=1$,
$\varepsilon_1\varepsilon_2=\varepsilon_1\varepsilon_3=\varepsilon_2\varepsilon_3=0$,
\item $\varepsilon_1[a,x]=\lambda\varepsilon_1[b,x]$, $\varepsilon_2[a,x]=0$
and $\varepsilon_3[b,x]=0$ for all $x\in R$.
\end{itemize}
That is $\varepsilon_1a-\lambda\varepsilon_1b, ~\varepsilon_2a,
~\varepsilon_3b \in C$. Then for all $x\in R$
$[\varepsilon_1b,[\varepsilon_1b,x]]=0$ since $\lambda$ is
invertible. By Lemma \ref{remark-bresar-ours} we obtain
$2\varepsilon_1b, ~\varepsilon_1b^2 \in C$. \end{proof}

\begin{lem}Let $R$ be a semiprime ring with extended centroid $C$ and $a, b\in R$. Then
$[[a,x],[b,x]]=0$ holds for all $x\in R$ if and only if there exist
an idempotent $\varepsilon \in C$ and an element $\zeta \in C$ such
that  $\varepsilon a-\zeta\varepsilon b, (1- \varepsilon)b\in C$.
\label{lemma-main}
\end{lem}

\begin{proof} The `` if " part is obvious. Now we deal with the ``
only if " part. Firstly, we will prove $[a,b]=0$. For any $x,y\in
R$, we get $ \big[[a,x+y],[b,x+y]\big]=0. $ Then for any $x,y\in R$
\begin{align}
\big[[a,x],[b,y]\big]+\big[[a,y],[b,x]\big]=0.\label{linearized}
\end{align}
Put $x=xb$ in \eqref{linearized}. Then for any $x,y\in R$
\begin{align*}
\big[x[a,b]+[a,x]b,[b,y]\big]+\big[[a,y],[b,x]b\big]=0.
\end{align*}
That is for any $x,y\in R$
\begin{align*}\begin{array}{rl}
x\big[[a,b],[b,y]\big]+\big[x,[b,y]\big][a,b]&+[a,x]\big[b,[b,y]\big]+\big[[a,x],[b,y]\big]b\\
&+[b,x]\big[[a,y],b\big]+\big[[a,y],[b,x]\big]b=0.
\end{array}\end{align*}
Then by \eqref{linearized} for any $x,y\in R$
\begin{align}
x\big[[a,b],[b,y]\big]+\big[x,[b,y]\big][a,b]&+[a,x]\big[b,[b,y]\big]
+[b,x]\big[[a,y],b\big]=0. \label{equation-main}
\end{align}
Put $y=b$ in \eqref{equation-main}. Then for any $x\in R$
\begin{align}
[b,x]\big[[a,b],b\big]=0. \label{equation-lemma1-1}
\end{align}
Putting $x=xy$ in \eqref{equation-lemma1-1} and applying
\eqref{equation-lemma1-1} we have that
\begin{align}
[b,x]y\big[[a,b],b\big]=0\label{equation-lemma1-2}
\end{align}
holds for all $x,y\in R$. Putting $x=-[a,b]$ in
\eqref{equation-lemma1-2}, we get that
$$\big[[a,b],b\big]y\big[[a,b],b\big]=0$$ holds for all $y\in R$. Then
$\big[[a,b],b\big]=0$ since $R$ is a semiprime ring. Putting $y=a$
into \eqref{equation-main} and applying $\big[[a,b],b\big]=0$ we
obtain that
\begin{align}
\big[x,[b,a]\big][a,b]=0\label{equation-lemma1-3}
\end{align}
holds for all $x\in R$. Putting $x=xy$ into
\eqref{equation-lemma1-3} and applying \eqref{equation-lemma1-3} we
get that $ \big[x,[b,a]\big]y[a,b]=0 $ holds for all $x,y\in R$.
Particularly  $ \big[x,[b,a]\big]y\big[x,[b,a]\big]=0$ for any
$x,y\in R$. Then $[a,b]\in Z(R)$ since $R$ is  semiprime. By
$\big[[a,ab],[b,ab]\big]=0$ and $[a,b]\in Z(R)$ we find
$-[a,b]^3=0$. Then $[a,b]=0$ since $[a,b]\in Z(R)$ and $R$ is
semiprime.

Review \eqref{equation-main}  then for any $x,y\in R$
\begin{align}
[a,x]\big[b,[b,y]\big] +[b,x]\big[[a,y],b\big]=0.
\label{equation-main-1}
\end{align}
Putting $x=xz$ in \eqref{equation-main-1} and applying
\eqref{equation-main-1} we obtain that
\begin{align}
[a,x]z\big[b,[b,y]\big] +[b,x]z\big[[a,y],b\big]=0
\label{equation-main-4}
\end{align}
holds for all $x,y,z\in R$. Putting $x=[b,x]$ in
\eqref{equation-main-4} and applying
$\big[[a,y],b\big]=-\big[a,[b,y]\big]$ (because of $[a,b]=0$), we
have that
$$\big[a,[b,x]\big]z\big[b,[b,y]\big] =\big[b,[b,x]\big]z\big[a,[b,y]\big]$$
holds for all $x,y,z\in R$. Then by Bre\v{s}ar Theorem \cite[Theorem
3.1]{1994-Bresar-ProcAMS} there exist idempotents
$\omega_1,\omega_2,\omega_3\in C$ and an invertible element $\xi\in
C$ such that
\begin{itemize}
\item $\omega_1+\omega_2+\omega_3=1$,
$\omega_1\omega_2=\omega_1\omega_3=\omega_2\omega_3=0$,
\item $\omega_1[a,[b,x]]=\xi\omega_1[b,[b,x]]$, $\omega_2[a,[b,x]]=0$
and $\omega_3[b,[b,x]]=0$  for all $x\in R$.
\end{itemize}
 Putting $x=-[a,y]$ in
\eqref{equation-main-4} and then multiplying  it by $\omega_3$, we
get that $$\omega_3\big[[a,y],b\big]z\omega_3\big[[a,y],b\big]=0$$
holds for all $y,z\in R$. So
$\omega_3\big[a,[b,y]\big]=-\omega_3\big[[a,y],b\big]=0$ since $R$
is semiprime. Hence \begin{center}
$\big[\omega_1a-\xi\omega_1b,[b,x]\big]=0$ and
$\big[(1-\omega_1)a,[b,x]\big]=0$
\end{center}
hold for all $x\in R$. Thus $\big[a-\xi\omega_1b,[b,x]\big]=0$ holds
for all $x\in R$. Then by Lemma \ref{lemma1} there exist
  idempotents $\varepsilon_1,
\varepsilon_2, \varepsilon_3 \in C$  and  an invertible element
$\lambda \in C$ such that
\begin{itemize}
\item $\varepsilon_1+\varepsilon_2+\varepsilon_3=1$,
$\varepsilon_1\varepsilon_2=\varepsilon_1\varepsilon_3=\varepsilon_2\varepsilon_3=0$
and
\item $\varepsilon_1(a-\xi\omega_1b)-\lambda\varepsilon_1b=c_1, ~\varepsilon_2(a-\xi\omega_1b)=c_2,
~\varepsilon_3b, ~2\varepsilon_1b, ~\varepsilon_1b^2 \in C$.
\end{itemize}
Then $$(\varepsilon_1+\varepsilon_2)a=
(\lambda\varepsilon_1+\xi\omega_1(\varepsilon_1+\varepsilon_2))(\varepsilon_1+\varepsilon_2)b+c_1+c_2
.$$ Setting $\varepsilon=\varepsilon_1+\varepsilon_2$ and
$\zeta=\lambda\varepsilon_1+\xi\omega_1(\varepsilon_1+\varepsilon_2)$,
we complete the proof.
\end{proof}

\begin{thm}
A skew $n$-derivation ($n\geq3$)  on a semiprime ring $R$ must map
into the center of $R$. \label{theorem-main}
\end{thm}

\begin{proof} Let $\Delta$ be a skew $n$-derivation on $R$ with
respect to the automorphism $\sigma$. Then for  fixed
$a_1,\ldots,a_n\in R$ we obtain
$\Delta(a_1,\ldots,a_n)=\Delta_1(a_1,a_2,a_3)$ where
$\Delta_1(x,y,z)=\Delta(x,y,z,a_4,\ldots,a_n)$ is a skew
3-derivation  with respect to  $\sigma$. So it is sufficient to
prove that every skew 3-derivation on $R$ must map into the center
of $R$. Let $\Delta:R\times R\times R\rightarrow R$ be a skew
3-derivation with respect to  $\sigma$. For fixed $x_0, y_0, z_0\in
R$, we proceed to prove that $\Delta(x_0,y_0,z_0)\in Z(R)$.
Obviously
\begin{center}
$\Delta(x_0,y,z)=\varphi_{x_0}(y,z)$,
~$\Delta(x,y_0,z)=\varphi_{y_0}(x,z)$ ~and
~$\Delta(x,y,z_0)=\varphi_{z_0}(x,y)$
\end{center}
are all $\sigma$-biderivations on $R$. Then by Remark
\ref{remark-biderivation} for every $t\in \{x_0,y_0,z_0\}$ there
exist idempotents $\varepsilon_t, \varepsilon_t', \varepsilon_t''\in
C$ and  invertible elements $p\in Q_s$, $\lambda_t\in C$ such that
\begin{itemize}
\item $\varepsilon_t+\varepsilon_t'+\varepsilon_t''=1$,
$\varepsilon_t\varepsilon_t'=\varepsilon_t\varepsilon_t''=\varepsilon_t'\varepsilon_t''=0$,
\item $\varepsilon_t\varphi_{t}(r,s)=\lambda_t\varepsilon_tp[r,s]$,
$\varepsilon_t'\varphi_{t}(r,s)=0$ and $\varepsilon_t''[r,s]=0$ for
all $r,s\in R$.
\end{itemize}
So for all $z\in R$, we obtain
$$\varepsilon_{x_0}\Delta(x_0,y_0,z)=\lambda_{x_0}\varepsilon_{x_0}p[y_0,z]
~~\text{and}
~~\varepsilon_{y_0}\Delta(x_0,y_0,z)=\lambda_{y_0}\varepsilon_{y_0}p[x_0,z].$$
Then for all $z\in R$, we have
$$\lambda_{x_0}\varepsilon_{x_0}\varepsilon_{y_0}[y_0,z]=
\varepsilon_{x_0}\varepsilon_{y_0}p^{-1}\Delta(x_0,y_0,z)=\lambda_{y_0}\varepsilon_{x_0}\varepsilon_{y_0}[x_0,z].$$
Hence for all $z\in R$, we get
$\lambda_{x_0}\varepsilon_{x_0}\varepsilon_{y_0}\big[[y_0,z],[x_0,z]\big]=0$.
Then by Lemma \ref{lemma-main}  we obtain
$\lambda_{x_0}\varepsilon_{x_0}\varepsilon_{y_0}[x_0,y_0]=0$. Thus
$\varepsilon_{x_0}\varepsilon_{y_0}[x_0,y_0]=0$ since
$\lambda_{x_0}$ is invertible. Then
$$\varepsilon_{x_0}\varepsilon_{y_0}\varepsilon_{z_0}\Delta(x_0,y_0,z_0)
=\varepsilon_{x_0}\varepsilon_{y_0}(\lambda_{z_0}\varepsilon_{z_0}p[x_0,y_0])=
\lambda_{z_0}\varepsilon_{z_0}p(\varepsilon_{x_0}\varepsilon_{y_0})[x_0,y_0]=0.$$
Set
$$\left\{\begin{array}{lcl}\varepsilon_1&=&\varepsilon_{x_0}\varepsilon_{y_0}(1-\varepsilon_{z_0}'')
+\varepsilon_{x_0}'(1-\varepsilon_{y_0}')+\varepsilon_{y_0}',\\
\varepsilon_2&=&\varepsilon_{x_0}(\varepsilon_{y_0}\varepsilon_{z_0}''+\varepsilon_{y_0}'')
+\varepsilon_{x_0}''(1-\varepsilon_{y_0}').
\end{array}\right.$$
It  can be verified from direct computation that $\varepsilon_1,
\varepsilon_2\in C$ are idempotents such that
\begin{itemize}
\item $\varepsilon_1+\varepsilon_2=1$,
\item $\varepsilon_1\Delta(x_0,y_0,z_0)=0$ ~and ~$\varepsilon_2[x,y]=0$ for all $x,y\in R$.
\end{itemize}
So for all $w\in R$, we have
$$[\Delta(x_0,y_0,z_0),w]=\varepsilon_1[\Delta(x_0,y_0,z_0),w]+\varepsilon_2[\Delta(x_0,y_0,z_0),w]=0.$$
Then $\Delta(x_0,y_0,z_0)\in Z(R)$ completes the proof. \end{proof}

 By Theorem \ref{theorem-main} and \cite[Theorem 3.2]{bresar-1995-JAlgebra-prime-genera} we get the following
result for prime rings.
\begin{thm}
A  prime ring with a nonzero skew $n$-derivation ($n\geq3$)
 must be commutative.
\end{thm}

\begin{proof} Let $\Delta$ be a nonzero  skew $n$-derivation
($n\geq3$) on a noncommutative prime ring $R$ with respect to an
automorphism $\sigma$. Then there exist $a_3,\ldots,a_n\in R$ such
that $\Delta_1(x,y)=\Delta(x,y,a_3,\ldots,a_n)$ is a nonzero
$\sigma$-biderivation on $R$. Then by  Theorem \ref{theorem-main}
and \cite[Theorem 3.2]{bresar-1995-JAlgebra-prime-genera} there
exists  an invertible element $p\in Q_s$ such that $[p[x,y],z]=0$
holds for all $x,y,z\in R$.  Particularly for all $x,y,z\in R$
$$0=[p[x,yx],z]=[p[x,y]x,z]=p[x,y][x,z].$$
Moreover for all $x,y,z\in R$ we have $[x,y]R[x,z]=0$ since $p$ is
invertible. So $R$ is commutative since $R$ is prime. \end{proof}
\bibliographystyle{amsplain}

\begin{thebibliography}{123}
\bibitem{Ali-2006-Aligarh-prime}Ali, A. and Kumar, D., {\it  Ideals and
symmetric ($\sigma,\sigma$)-biderivations on prime rings},
    Aligarh Bull. Math., {\bf 25}(2006), 9--18.


\bibitem{Argac-2006-AlgebraColloq-prime}Argac, N., {\it On prime and semiprime rings with
derivations},
    Algebra Colloq., {\bf 13}(2006), 371--380.

\bibitem{Argac-1992-SymmSci-prime}Argac, N. and Yenigul, M.S., {\it Lie ideals and symmetric bi-derivations of prime rings},
     Symmetries in Science, VI (Bregenz, 1992),  41--45, Plenum,  New York, 1993.


\bibitem{Argac-1996-PureAppl-prime}Argac, N. and Yenigul, M.S., {\it Lie ideals and
symmetric bi-derivations of prime rings},
     Pure
Appl. Math. Sci., {\bf 44}(1996), 17--21.


\bibitem{Ashraf-1999-Rend-prime}Ashraf, M., {\it On symmetric bi-derivations in rings},
     Rend. Istit. Mat. Univ. Trieste, {\bf 31}(1999), 25--36.

\bibitem{Ashraf-1997-8-Aligarh-prime}Ashraf, M. and Rehman, N., {\it On symmetric
($\sigma,\sigma$)-biderivations},
     Aligarh Bull. Math., {\bf 17}(1997/98), 9--16.



\bibitem{book}Beidar, K.I.,  Martindale 3rd, W.S. and  Mikhalev, A.V., {\it    Rings with Generalized
Identities},
       Marcel Dekker, Inc., New York (1996).

\bibitem{1994-Bresar-ProcAMS}Bre\v{s}ar, M., {\it On certain pairs of functions of semiprime rings}, Proc. Amer. Math.
Soc.,
  {\bf 120}(1994), 709-713.


\bibitem{bresar-1995-JAlgebra-prime}Bre\v{s}ar, M., {\it Functional identities of degree two},
     J. Algebra, {\bf 172}(1995), 690--720.


\bibitem{bresar-1995-JAlgebra-prime-genera}Bre\v{s}ar, M., {\it On generalized biderivations and related
maps},
     J. Algebra, {\bf 172}(1995), 764--786.

\bibitem{2004-Bresar-Taiwan}Bre\v{s}ar, M., {\it Commuting maps: a survey}, Taiwanese  J. Math.,
  {\bf 8}(2004), 361-397.


\bibitem{Ceran-2009-Algebras-prime}Ceran, S. and Asci, M., {\it On traces of symmetric bi-($\sigma,\tau$)
derivations on prime rings},
    Algebras Groups Geom., {\bf 26}(2003), 203--214.


\bibitem{Deng-1996-JMathResExp-prime}Deng, Q., {\it Symmetric biderivations and
commutativity of prime rings},
     J. Math. Res. Exposition, {\bf 16}(1996), 427--430.

\bibitem{Deng-1997-InternatJMath-prime}Deng, Q., {\it On a conjecture of Vukman},
     Internat. J. Math. Math. Sci., {\bf 20}(1997), 263--266.

\bibitem{Fosner-submit-to-algebraCollo}Fo\v{s}ner, A., {\it On generalized derivations},
     submitted to Algebra
Colloquium.



\bibitem{Hongan-2000-CommAlg-prime}Hongan, M. and Komatsu, H., {\it On the module of
differentials of a noncommutative algebra and symmetric
biderivations of a semiprime algebra},
     Comm. Algebra, {\bf 28}(2000), 669--692.











\bibitem{Huang-2007-study-prime}Huang, S.L. and Fu, S.T., {\it Remarks on generalized
derivations and symmetric biderivations of prime rings},
    J. Math. Study, {\bf 40}(2007), 360--364.



\bibitem{Jung-2007-BKMS-3}Jung, Y.S. and Park, K.H., {\it On prime and semiprime rings
with permuting 3-derivations},
    Bull. Korean Math.
Soc., {\bf 44}(2007), 789--794.


\bibitem{Kannappan-1994-Glas-prime}Kannappan, P., {\it Jordan derivation and functional equations},
     Glas. Mat. Ser. III, {\bf 29(49)}(1994), 305--310.

 \bibitem{Khan-2003-Southeas-prime}Khan, M.A., {\it Remarks on symmetric biderivations of rings},
    Southeast Asian Bull. Math., {\bf 27}(2003), 631--640.


 \bibitem{Kharchenko-1991-SMJ-auto-semi}Kharchenko, V.K., {\it Skew derivations of semiprime rings},
    Siberian Math. J., {\bf 32}(1991), 1045--1051.


\bibitem{Lanski1997Engel}Lanski, C., {\it An Engel condition with derivation for left
ideal},
        Proc. Amer. Math. Soc.,  {\bf 125}(1997), 339--345.


\bibitem{LeeTK-1992-Semi}Lee, T.K., {\it Semiprime rings with differential identities},
        Bull. Inst. Math. Acad. Sinica,  {\bf 20}(1992), 27--38.




\bibitem{Maksa-1980-GlasMS}Maksa, G., {\it A remark on symmetric biadditive functions having nonnegative diagonalization},
        Glas. Mat. Ser. III,  {\bf 15(35)}(1980), 279--282.

\bibitem{Muthanna-2007-Aligarh-prime}Muthana, N.M., {\it Orthogonality of traces and
derivations in semiprime rings},
    Aligarh Bull.
Math., {\bf 26}(2007), 49--60.

\bibitem{Park-2007-JKS-4}Park, K.H., {\it On 4-permuting 4-derivations in prime and semiprime
rings},
    J. Korea Soc. Math. Educ. Ser. B Pure Appl.
Math., {\bf 14}(2007), 271--278.


\bibitem{Park-2009-JCCM-4}Park, K.H., {\it On prime and
semiprime rings with symmetric n-derivations},
   J. Chungcheong Math.
Soc., {\bf 22}(2009), 451--458.


\bibitem{Posner-1957-PAMS}Posner, E., {\it Derivations in prime rings},
    Proc. Amer. Math. Soc.,  {\bf 8}(1957), 1093--1100.




\bibitem{Vukman-1989-AequaMath}Vukman, J., {\it Symmetric bi-derivations on prime and semi-prime rings},
      Aequationes Math.,  {\bf 38}(1989), 245--254.


\bibitem{Vukman-1990-AequaMath-prime}Vukman, J., {\it Two results  concerning symmetric bi-derivations on prime  rings},
      Aequationes Math.,  {\bf 40}(1990), 181--189.

\bibitem{Vukman-1990-GlasMathS-general}Vukman, J., {\it A functional equation in rings},
      Glas. Mat. Ser. III,  {\bf 25(45)}(1990), 329--334.



\bibitem{WangY-2002-CommAlg-prime}Wang, Y., {\it Symmetric bi-derivation of prime rings},
     J. Math. Res. Exposition, {\bf 22}(2002), 503--504.

\bibitem{Yenigul-1993-MathJOkayamaMath-prime}Yenigul, M.S. and Argac, N., {\it Ideals and symmetric bi-derivations of prime and semi-prime rings},
     Math. J. Okayama Univ., {\bf 35}(1993), 189--192.


\bibitem{zhan-2003-Qufu-prime}Zhan, J.M., {\it T-symmetric bi-derivations of prime rings},
     Qufu Shifan Daxue Xuebao Ziran
Kexue Ban, {\bf 29}(2003), 16--18.
\end{thebibliography}

\end{document}